\documentclass[a4paper, english, 11pt, twoside]{amsart}

\usepackage{stmaryrd,amscd,amsmath,amsbsy,hyperref}
\usepackage{amssymb,mathrsfs}
\usepackage{amsfonts}
\usepackage{xcolor}
\usepackage{bbold}
\usepackage[notcite,notref,final]{showkeys}

\newtheorem{theorem}{Theorem}[section]
\newtheorem{lemma}[theorem]{Lemma}
\newtheorem{proposition}[theorem]{Proposition}

\theoremstyle{definition}

\newtheorem{hypothesis}[theorem]{Hypothesis}

\newtheorem{example}[theorem]{Example}

\newtheorem{remark}[theorem]{Remark}

\numberwithin{equation}{section}
\makeatletter              
\let\c@equation\c@theorem  
\makeatother

\newcommand{\FF}{{\mathbb{F}}}
\newcommand{\QQ}{{\mathcal{Q}}}

\begin{document}

\title[Finite dimensional Hopf actions on Weyl algebras]{Finite dimensional Hopf actions on Weyl algebras}

\author{Juan Cuadra}
\address{Department of Mathematics, University of Almer\'{i}a, 04120 Almer\'{i}a,
Spain}
\email{jcdiaz@ual.es}

\author{Pavel Etingof}
\address{Department of Mathematics, Massachusetts Institute of Technology,
Cambridge, Massachusetts 02139, USA}
\email{etingof@math.mit.edu}

\author{Chelsea Walton}
\address{Department of Mathematics, Temple University,
Philadelphia, Pennsylvania 19122, USA}
\email{notlaw@temple.edu}


\maketitle

\begin{abstract}
We prove that any action of a finite dimensional Hopf algebra $H$ on a Weyl algebra $A$ over an algebraically closed field of characteristic zero factors through a group action. In other words, Weyl algebras do not admit genuine finite quantum symmetries. This improves a previous result by the authors, where the statement was established for semisimple $H$. The proof relies on a refinement of the method previously used: namely, considering reductions of the action of $H$ on $A$ modulo prime powers rather than primes. We also show that the result holds, more generally, for algebras of differential operators. This gives an affirmative answer to a question posed by the last two authors.
\end{abstract}

\maketitle

\section{Introduction}

Let $k$ be an algebraically closed field of characteristic zero.  In \cite[Theorem 1.3]{EW1}, it is shown that any action of a semisimple Hopf algebra $H$ on a commutative domain over $k$ factors through a group action. In particular, if the action is inner faithful, i.e., does not factor through that of a Hopf algebra of smaller dimension, then $H$ is a group algebra. \par \smallskip

As an application of this result, it is proved in \cite[Corollary 5.5]{EW1} that if $H$ acts on $\bold A_n(k)$, the $n$-th Weyl algebra over $k$, and the action preserves the standard filtration, then the action factors through a group action. The idea is to use the associated graded algebra. \par \smallskip

This result was complemented in \cite[Theorem 4.1]{CEW} with a similar statement, but replacing the stability of the filtration by the semisimplicity of $H$. The strategy in this case is different. The idea is to reduce the action to positive characteristic, where $\bold A_n(k)$ becomes an Azumaya algebra over its center, and then pass it to the division ring of quotients. The center of the latter is stabilized by the action and \cite[Theorem 1.3]{EW1} is used again. \par \smallskip

The goal of this paper is to prove the desired unconditional statement:

\begin{theorem}\label{main}
Any action of a finite dimensional Hopf algebra $H$ on $\bold A_n(k)$ factors through a  group action.
\end{theorem}

In particular, if the action is inner faithful, Theorem \ref{main} implies that $H$ must be a group algebra. In other words, the Weyl algebra has no genuine finite quantum symmetries. \par \smallskip

The proof of Theorem \ref{main} uses ideas from that of \cite[Theorem 4.1]{CEW}, but differs from it in several important ways:
\begin{enumerate}
\item The proof uses reduction modulo prime powers and not just modulo primes; \vspace{3pt}
\item  The proof does {\it not} use the main result of \cite{EW1} (cf. \cite[proof of Proposition 3.3(ii)]{CEW}); \vspace{3pt}
\item Unlike \cite{CEW}, the proof (and in fact, the theorem itself) fails  when $\bold A_n(k)$ is replaced by $\bold A_n(k[z_1,\dots,z_s])$, see \cite[Proposition 4.3]{CEW}. This happens even for $n=0$, as there are inner faithful actions of non-semisimple Hopf algebras on polynomial algebras, see for example \cite{EW2} and references therein.
\end{enumerate}
\par \smallskip

At the end of the paper, we show that our result extends to finite dimensional Hopf actions on algebras of differential operators. More precisely, we prove:

\begin{theorem} \label{diffop}
Let $D(X)$ be the algebra of differential operators on a smooth affine irreducible variety $X$ over $k$. Then, any finite dimensional Hopf action on $D(X)$ factors through a group action.
\end{theorem}

Theorem \ref{main} is a special case of Theorem~\ref{diffop}, for $X=\Bbb A^n$. Theorem~\ref{diffop} gives an affirmative answer to \cite[Question 5.7]{EW1}, even without the assumption on the stability of the filtration. \par \smallskip

Arguing as in the proof of \cite[Proposition 4.4]{CEW}, one can show that Theorem \ref{diffop} remains valid when $D(X)$ is replaced by its division ring of quotients ${\mathcal Q}_{D(X)}$. \par \smallskip

It would be interesting to establish absence of genuine finite quantum symmetries for more general classes of noncommutative algebras. This is the subject of future work. We refer the reader to \cite{Kirkman} for an account on recent developments in the study of Hopf actions on some natural classes of noncommutative algebras. \par \smallskip

The paper is organized as follows. Preliminary results on invariants of Hopf actions on Weyl division algebras and on reduction modulo prime powers are provided in Section~\ref{prelim}. In Section~\ref{towersec}, we establish an auxiliary result on Hopf actions on fields in positive characteristic. The proofs of Theorems~\ref{main} and~\ref{diffop} are given in Section~\ref{proofmain}.

\section{Preliminary results} \label{prelim}

Unless stated otherwise, we will use the definitions and results from \cite{CEW} throughout the paper. We recall the notation, assumptions and some facts from there: \medskip

\noindent $\bullet$ $k$ is an algebraically closed field of characteristic zero; \smallskip

\noindent $\bullet$ $H$ is a finite dimensional Hopf algebra over $k$; \smallskip

\noindent $\bullet$ $A$ is the $n$-th Weyl algebra $\bold A_n(k)$ generated by $x_i$, $y_i$ for $i = 1, \ldots ,n$, subject to relations $[x_i,x_j]=[y_i,y_j]=0$ and $[y_i,x_j]=\delta_{ij}$. We assume that it carries an inner faithful $H$-action $\cdot : H \otimes_k A\to A$; \smallskip

\noindent $\bullet$ $R$ is a finitely generated subring of $k$ containing the structure constants of $H$ and those of the $H$-action; \smallskip

\noindent $\bullet$ $H_R$ is a Hopf $R$-order of $H$, so that the multiplication by scalars induces an isomorphism $H_R \otimes_R k \cong H$. The $H$-action restricts to an action \linebreak $\cdot_R:H_R \, \otimes_R A_R \to A_R$, with $A_R:=\bold A_n(R)$. See \cite[Lemma 2.2]{CEW}; \smallskip

\noindent $\bullet$ $H_p:=H\otimes_R \overline{\Bbb F}_p$ is the reduction of $H$ modulo a sufficiently large prime number $p$, associated to a homomorphism $\psi: R\to \overline{\Bbb F}_p$. See \cite[Lemma~2.3]{CEW};

\smallskip

\noindent $\bullet$ $A_p:=\bold A_n(\overline{\mathbb{F}}_p)$ is the reduction of $A$ modulo $p$. By tensoring the $H_R$-action $\cdot_R:H_R \otimes A_R\to A_R$ with $\overline{\FF}_p$ over $R$ we endow $A_p$ with an inner faithful $H_p$-action $\cdot_p: H_p \otimes_{\overline{\FF}_p} A_p \to A_p$, see \cite[Proposition 2.4]{CEW}; \smallskip

\noindent $\bullet$ $D_p$ is the full localization of $A_p$, a division algebra over $\overline{\mathbb{F}}_p$, which, by \cite[Lemma 3.1]{CEW} carries an inner faithful action of $H_p$ induced from that on $A_p$; and \smallskip

\noindent $\bullet$ $Z$ is the center of $D_p$. We will see in the Proposition \ref{propCEW} below that $Z$ is $H_p$-stable. (Note that we do not indicate the dependence of $Z$ on $p$ here, as the prime $p$ is fixed.)
\medskip

In the rest of the section, we provide results on invariants of Hopf actions on division algebras and reduction modulo prime powers, both of which we will use to prove Theorem \ref{main}. But first we discuss a version of Hensel's lemma needed for this work.

\subsection{Witt vectors and Hensel's lemma}

Let us recall here some basic facts from commutative algebra and algebraic geometry. \par \smallskip

Let $W_p$ be {\it the ring of Witt vectors of  $\overline{\Bbb F}_p$}; see \cite[Section II.6]{Se}. Let
$$
W_{m,p}:=W_p/(p^m)
$$
be {\it the $m$-truncated ring of Witt vectors of $\overline{\Bbb F}_p$}, which is an algebra over the ring $\Bbb Z/p^m\Bbb Z$. \footnote{In the sequel we consider several algebras over $\Bbb Z/p^m\Bbb Z$, and it is important to remember that they are not vector spaces over a field but only modules over the ring $\Bbb Z/p^m\Bbb Z$.} \par \smallskip

For sufficiently large primes $p$, we have that $R/(p)\ne 0$. Further, $R$ is unramified at $p$ so that the algebra
$R_p:=R\otimes_{\Bbb Z}\overline{\Bbb F}_p$ has no nonzero nilpotent elements. Thus,
$X_p:={\rm Spec}(R_p)$ is a nonempty algebraic variety over $\overline{\Bbb F}_p$. \par \smallskip

If $\psi: R\to \overline{\Bbb F}_p$
is a smooth point of $X_p$ (which is the case for generic $\psi$), then we have the following version of Hensel's lemma.

\begin{lemma}\label{Hensel} The point
$\psi$ can be lifted to a point $\psi_\infty$ of ${\rm Spec}(R)$ over the ring of Witt vectors $W_p$.
In other words, there exists an (in general, non-unique) homomorphism $\psi_\infty: R\to W_p$ whose reduction modulo $p$
gives $\psi$.
\end{lemma}

Let us choose such a lifting $\psi_\infty$; then, reducing it modulo $p^m$, we obtain
homomorphisms
$$
\psi_m: R\to W_{m,p},
$$
such that $\psi_m$ equals the reduction of $\psi_{m+1}$
modulo $p^m$ for $m\ge 1$, and $\psi_1=\psi$.

\subsection{Auxiliary results on division algebras}

The following result collects \cite[Proposition 3.3(i)]{CEW} and a fact contained in the proof of \cite[Proposition 3.3(ii)]{CEW}. 

\begin{proposition} \label{propCEW}
Let $H$ be a finite dimensional Hopf algebra over an algebraically closed field $F$ of dimension $d$. Let $D$ be a division algebra over $F$ of degree $m$, that admits an action of $H$. If $\gcd(d!,m)=1$, then:
\begin{itemize}
\item[(i)] The center $Z$ of $D$ is $H$-stable, and $D=ZD^H$.
\item[(ii)] $H$ acts inner faithfully on $Z$. \qed
\end{itemize}
\end{proposition}

Return to the notation set at the beginning of the section. The next three results pertain to the quotient field $\mathcal{Q}_A$ of $A$.

\begin{lemma}\label{locali1}
Let $S$ be a left Ore domain, $\QQ_S$ its
division ring of fractions, and let $C\subset \QQ_S$ be a division subalgebra such that $CS$ is finite dimensional as a left $C$-vector space. Then, $CS=\QQ_S$.
In particular, this holds if $\QQ_S$ is finite dimensional as a left $C$-vector space.
\end{lemma}

\begin{proof}
Any element of $\QQ_S$ can be written as $g^{-1}f$,
for $f,g\in S$, so it suffices to show that $g^{-1}\in CS$.
To this end, note that since $CS$ is finite dimensional over $C$,
the element $g$ must satisfy a polynomial equation over $C$:
$$
a_0g^r+a_1g^{r-1}+\dots+a_r=0, \quad \textrm{with} \ a_i\in C.
$$
Without loss of generality, we may assume that $a_r\ne 0$ (otherwise we can multiply on the right by an appropriate negative power of $g$). Hence, \linebreak $g^{-1}=-a_r^{-1}\sum_{i=0}^{r-1}a_ig^{r-i-1}\in CS$.
\end{proof}

Recall that the $H$-action on $A$ extends uniquely to the quotient division algebra $\QQ_A$ of $A$ by \cite[Theorem 2.2]{SVO}. We apply Lemma \ref{locali1} to $S:=A $ and $C:=\QQ_A^H$.

\begin{lemma}\label{locali}
One has $\QQ_A^HA=\QQ_A$.
\end{lemma}

\begin{proof}
By \cite[Corollary 2.3]{BCF} (restated in \cite[Lemma 3.2]{CEW}),
the dimension of $\QQ_A$ over the division ring of invariants $\QQ_A^H$
(on either side) is less or equal than $\dim H$. The result now follows from Lemma \ref{locali1}.
\end{proof}

We next see that Lemma \ref{locali} allows us to choose a convenient finite spanning set
for $\QQ_A$ over $\QQ_A^H$ contained in $A$.  For a monomial $u$ in $x_i,y_i$, let $\deg(u)$ be the degree of $u$.

\begin{lemma} \label{express}
There exists a positive integer $N$ so that
we can express a monomial $v$ of any degree in $x_i, y_i$  as
$$
v=\sum_{u:\ \deg(u)\le N} b_{v,u} u,
$$ where
the $u$ are monomials in $x_i$, $y_i$, and the $b_{v,u}$ are noncommutative polynomials in
the finite set of elements $b_{w,u} \in \QQ_A^H$,
with $w$ having degree $N+1$.
\end{lemma}

\begin{proof}
Let $N$ be a positive integer such that the monomials in $x_1,\dots,x_n,$ $y_1,\dots,y_n$
of degree $\le N$ span $\QQ_A$ over $\QQ_A^H$ as a left vector space.
Such $N$ exists, since $\QQ_A$ is finite dimensional over $\QQ_A^H$ by \cite[Corollary~2.3]{BCF},
and by Lemma \ref{locali} is spanned over $\QQ_A^H$ by $A$. Then, for each monomial
$w$ in $x_i,y_i$ of degree $N+1$, we have
\begin{equation}\label{spanset}
w=\sum_{u:\ \deg(u)\le N} b_{w,u}u,
\end{equation}
where the $u$ are monomials in $x_i$, $y_i$, and $b_{w,u}\in \QQ_A^H$.
By applying \eqref{spanset} repeatedly, we obtain the result for $v$ of any degree; namely, $b_{v,u}$ is a noncommutative polynomial in the
finite set of elements $\{b_{w,u}\}$, $\deg(w)=N+1$, and $\deg(u)\le N$.
\end{proof}

We will also need the following lemma from the theory of  division algebras.
The lemma is well known, but we provide a proof for reader's convenience.

\begin{lemma}\label{divalg}
Let $D_1\subset D_2$ be division algebras each finite dimensional over its center, 
with $[D_2:D_1]<\infty$, and let the degree of $D_i$ be $d_i$ for $i=1,2$. Let $Z_i$ be the center of $D_i$. Then:
\begin{enumerate}
\item[(i)] $d_2/d_1$ is an integer dividing $[D_2:D_1]$;
\item[(ii)] If $d_1=d_2$, then $D_2=Z_2D_1 \cong Z_2\otimes_{Z_1}D_1$.
\end{enumerate}
\end{lemma}

\begin{proof}
(i) We have $[D_2:D_1]=[D_2:Z_2D_1][Z_2D_1:D_1]$.
The center of $Z_2D_1$ is some field $K$ containing $Z_1$ and $Z_2$. Thus $KD_1=Z_2D_1$. Moreover, we have $[Z_2D_1:K]=d_1^2$ because $K \otimes_{Z_1} D_1 \cong KD_1$. Let $[K:Z_2]=m$.
Then,
$$
\begin{array}{ll}
d_2^2 & = [D_2:Z_2] \vspace{2pt} \\
      & = [D_2 : Z_2D_1][Z_2D_1:Z_2] \vspace{2pt} \\
      & = [D_2 : Z_2D_1][Z_2D_1:K][K:Z_2] \vspace{2pt} \\
      & = [D_2 : Z_2D_1]d_1^2m.
\end{array}
$$
So $[D_2:Z_2D_1]=d_2^2/d_1^2m$. \par \smallskip

Let $L$ be a maximal subfield of $Z_2D_1$. Then $[L:K]=d_1$ and consequently  $[L:Z_2]=[L:K][K:Z_2]=d_1m$. Now, $L$ is contained in a maximal subfield $L'$ of $D_2$, with $[L':Z_2]=d_2$. So,
$d_2=[L':L]d_1m$. Thus, $d_1m$ divides $d_2$. Hence, $d_2/d_1$ is an integer dividing
$[D_2:Z_2D_1]=d_2^2/d_1^2m$, which in turn divides $[D_2:D_1]$.  \par \medskip

(ii) If $d_1=d_2$, then $[D_2 : Z_2D_1]=1$, and so $D_2=Z_2D_1$. Thus, $Z_1\subset Z_2$ and $D_2 \cong Z_2\otimes_{Z_1}D_1$.
\end{proof}

\subsection{Reduction modulo prime powers}

Now we want to reduce the action of $H$ on $A$ modulo prime powers.
This is done in a standard way, as one does for any kind of ``finite" linear algebraic structures.
The process is similar to reduction modulo a prime described in \cite[Section 2]{CEW},
but somewhat more complicated. \par \smallskip

Recall from \cite[Lemma 2.2]{CEW} that the algebras $H$, $A$ and the action of $H$ on $A$ are defined over some finitely generated subring $R\subset k$. We have the corresponding $R$-orders $H_R$ and $A_R$ and the restricted action of $H_R$ on $A_R$. For a sufficiently large prime $p$, fix a smooth point $\psi\in X_p$ and its lifting $\psi_\infty$ to $W_p$, which gives
rise to the maps $\psi_m$, $m\ge 1$ (see Lemma \ref{Hensel}). Now we define {\it reductions of $H$ and $A$ modulo $p^m$} by the formulas:
\medskip

\noindent $\bullet$ $H_{p^m}=H_R\otimes_R W_{m,p}$, \ and
\smallskip

\noindent $\bullet$ $A_{p^m}=A_R\otimes_R W_{m,p}$. \par \smallskip

\noindent Thus, in the notation, we suppress the dependence of these reductions on the choice of $\psi_m$.
Note that $A_{p^m}=\bold A_n(W_{m,p})$.

Similar to \cite[Proposition~2.4]{CEW}, $H_{p^m}$ acts on $A_{p^m}$  by tensoring the action of $H_R$ on $A_R$ with $W_{m,p}$ over $R$ using $\psi_m$.

\subsection{The ring $D_{p^m}$ and its center $Z_m$.}\label{centerst}
We define: \medskip

\noindent $\bullet$ $D_{p^m}$ as the full localization of $A_{p^m}$, \ and
\smallskip

\noindent $\bullet$ $Z_m$ as the center of $D_{p^m}$.

\medskip

\noindent The algebra $D_{p^m}$ is obtained from $A_{p^m}$ by inverting all elements which are not zero divisors, i.e., not contained in the ideal
$(p)$. Thus, $D_p$ is the noncommutative field of quotients of the Weyl algebra $A_p=\bold A_n(\overline{\Bbb F}_p)$ (as in \cite{CEW}). Further,  $D_{p^m}$ can be visualized as follows: its associated graded algebra under the filtration by powers of $p$ is ${\rm gr}(D_{p^m})=D_p[z]/(z^m)$. It is therefore easy to see that $D_{p^m}$ is an Artinian ring. \par \smallskip

Further, observe that $Z_m$ contains the ring of rational functions

\medskip

\noindent $\bullet$ $K_m:=W_{m,p}(x_i^{p^m},y_i^{p^m} : 1\le i\le n)$.

\medskip

\noindent (By a rational function we mean a fraction $P/Q$, where $P,Q$ are polynomials, and
$Q$ has a nonzero reduction modulo $p$). Moreover, $D_{p^m}$ is a free module
over $K_m$ with basis consisting of ordered monomials $(\prod_i x_i^{\alpha_i})(\prod_i y_i^{\beta_i})$,
where $\alpha=(\alpha_i)$ and $\beta=(\beta_i)$ are multi-indices, such that $0\le \alpha_i,\beta_i\le p^m-1$; the rank of this module is $p^{2nm}$. \par \smallskip

The structure of $Z_m$ is described by the following result.

\begin{lemma}\label{centerstr}
The center $Z_m$ of $D_{p^m}$ is spanned over $K_m$
by the elements
$$
v_{\alpha,\beta}:=p^{m-s(\alpha,\beta)}\Big(\prod_i x_i^{\alpha_i}\Big)\Big(\prod_i y_i^{\beta_i}\Big), \quad \text{for }\
0\le \alpha_i,\beta_i\le p^m-1,\ s(\alpha,\beta)>0,
$$
where $s(\alpha,\beta)$ is the largest integer such that $p^{s(\alpha,\beta)}$ divides ${\rm gcd}(\alpha_i,\beta_i)$ for $i=1,\dots,n$.
Moreover, the defining relations of $Z_m$ as a $K_m$-module on these generators
are $p^{s(\alpha,\beta)}v_{\alpha,\beta}=0$.
\end{lemma}

\begin{proof}
Take $f=\sum_{\alpha,\beta}c_{\alpha,\beta}(\prod_i x_i^{\alpha_i})(\prod_i y_i^{\beta_i}) \in D_{p^m}$,
where $c_{\alpha,\beta}\in K_m$. Commuting $f$ with $x_i$ and $y_i$, we find that $f\in Z_m$ if and only if
$p^{s(\alpha,\beta)}c_{\alpha,\beta}=0$ for all $\alpha,\beta$. This implies the statement.
\end{proof}

\section{Hopf actions on fields of characteristic $p$ preserving $p^m$-th powers}\label{towersec}

The following theorem plays an auxiliary role in this paper, but is of independent interest. Throughout this section, we make the following assumptions:

\begin{hypothesis} \label{hypsec3} Take $H$ to be a finite dimensional Hopf algebra over an algebraically closed field $F$ of characteristic $p$, and take $Z$ to be a finitely generated field extension of $F$. Assume that $H$ acts $F$-linearly and inner faithfully on $Z$. All algebras and $\otimes$ are over $F$. Let
$$Z^{p^m} := \{z^{p^m} : z \in Z\}.$$
\end{hypothesis}

The main result of this section is:

\begin{theorem}\label{tower}
Suppose that $p>\dim H$, and $H$ preserves $Z^{p^m}$ for all \linebreak $m\ge 1$. Then $H$ is a group algebra.
\end{theorem}

The proof of Theorem~\ref{tower} is provided at the end of this section. First, we need the following two lemmas pertaining to the
coideal subalgebra attached to the action of $H$ on $Z$.

Let $\rho: Z\to Z\otimes H^*$ be the dual coaction map. Consider the Galois map
$$
can: Z\otimes_{Z^H}Z\to Z\otimes H^*, \ z\otimes z' \mapsto (z\otimes 1)\rho(z').
$$
Let $B$ be the image of $can$. Then $B$ is a commutative coideal subalgebra in the Hopf algebra $Z\otimes H^*$ over $Z$. The commutativity is clear, and the coideal subalgebra condition follows from an argument similar to  \cite[Lemma 3.2]{EW1}. \par \smallskip

Moreover, we have the following:

\begin{lemma}\label{fieldofdef}
Suppose that $B$ is defined over $F$, that is to say, $B=Z\otimes B_0$ for some subalgebra $B_0\subset H^*$. Then, $B_0=H^*$ and $B=Z\otimes H^*$. In particular, $H$ is cocommutative.
\end{lemma}

\begin{proof}
Let $\{b_i\}_{i \in \mathcal{I}}$ be a basis of $B_0$, for some index set $\mathcal{I}$. Thus, the coaction of $H^*$ on $Z$ is defined by the formula
$$
\rho(z)=\sum_i \rho_i(z)\otimes b_i,
$$
for linear maps $\rho_i: Z\to Z$. Applying the coproduct in the second component and using coassociativity, we get
$$
\sum_i \rho_i(z)\otimes \Delta(b_i)=\sum_{i,j}\rho_j(\rho_i(z))\otimes b_j\otimes b_i.
$$
Let $a_{sm},z_{sm}\in Z$ be such that $\sum_s (a_{sm}\otimes 1)\rho(z_{sm})=1\otimes b_m$. They exist because $B=Z\otimes B_0$. Applying the coproduct again in the second component and using the
previous equality, we obtain
$$
1\otimes \Delta(b_m)=\sum_{i,s} a_{sm}\rho_i(z_{sm})\otimes \Delta(b_i)=\sum_{i,j,s}a_{sm}\rho_j(\rho_i(z_{sm}))\otimes b_j\otimes b_i.
$$
This implies that $\Delta(b_m)\in B_0\otimes B_0$. In other words, $B_0$ is a subbialgebra of $H^*$. Since $H^*$ is finite dimensional, $B_0$ is a Hopf subalgebra of $H^*$. Since $H$ acts inner faithfully on $Z$, there does not exist a proper Hopf subalgebra $K$ of $H^*$
so that $\rho(Z) \subset Z \otimes K$. Hence, $B_0=H^*$. But $B_0$ is commutative by assumption, so $H^*$ is commutative and $H$ is cocommutative, as desired.
 \end{proof}

\begin{lemma} \label{Plucker} The following conditions on $B$ are equivalent:
\begin{enumerate}
\item[(i)] $B$ is defined over $F$;
\item[(ii)] For any $m$, the subspace $B$ is defined over $Z^{p^m}$, that is to say, there exists an $Z^{p^m}$-subspace $V_m$ of $Z^{p^m}\otimes H^*$ such that $B=Z\otimes_{Z^{p^m}}V_m$.
\end{enumerate}
\end{lemma}

\begin{proof}
It is clear that the intersection $L:=\bigcap_{m\ge 0}Z^{p^m}$ is a perfect field. Also, $L$ is finitely generated over $F$, since it is a subfield of $Z$ containing $F$. This yields that $L=F$.

Now, condition (i) is equivalent to the condition that ratios of the Pl\"ucker coordinates of $B$ as a $Z$-subspace of $Z\otimes H^*$ lie in $F$.
Since $\bigcap_{m\ge 0}Z^{p^m}=F$, this is, in turn, equivalent to the condition that ratios of the Pl\"ucker coordinates of $B$ lie in $Z^{p^m}$ for all $m$.
But the last statement is clearly equivalent to condition (ii).
\end{proof}

Now we prove the main result of this section.
\medskip

\noindent {\it Proof  of Theorem \ref{tower}}.
For any $m$, let $V_m$ denote the span of $can(z\otimes z')$, where
$z,z'\in Z^{p^m}$. Since $H$ preserves $Z^{p^m}$, the space $V_m$ is a $Z^{p^m}$-subspace
of $Z^{p^m}\otimes H^*$. \par \smallskip

Now, we claim that $Z^HZ^{p^m}=Z$. Indeed, by \cite[Corollary 2.3]{BCF}, we have $[Z:Z^H]\le \dim H$, so $[Z:Z^H]$ is not divisible by $p$. Since $[Z:Z^HZ^{p^m}]$ divides $[Z:Z^H]$, we also have that $[Z:Z^HZ^{p^m}]$ is not divisible by $p$. On the other hand, $[Z:Z^HZ^{p^m}]$ divides
$[Z:Z^{p^m}]$, so is a power of $p$. Therefore, $[Z:Z^HZ^{p^m}]=1$. \par \smallskip

Thus, $Z \otimes_{Z^H} Z = Z\otimes_{Z^H}(Z^HZ^{p^m})=Z\otimes_{Z^H\cap Z^{p^m}} Z^{p^m}$. Hence,
$$B=can(Z \otimes_{Z^H} Z) = can\big(Z\otimes_{Z^H\cap Z^{p^m}} Z^{p^m}\big)$$
is equal to $Z\otimes_{Z^{p^m}}V_m$.
Hence, $B$ is defined over $Z^{p^m}$ for all $m$, which by Lemmas~\ref{fieldofdef} and~\ref{Plucker}, implies that $H$ is cocommutative.
Thus, $H$ is a group algebra (using again that $p>\dim H$). \qed

\begin{remark} By \cite[Proposition 3.9]{Et}, the assumption in Theorem \ref{tower} that $p>\dim H$ can be replaced by a weaker assumption that $p$ does not divide $\dim H$.
\end{remark}

 \section{Proof of Theorems \ref{main} and~\ref{diffop}} \label{proofmain}
 To begin, we simplify notation as follows.

 \medskip

\noindent {\it Notation}. We denote invariants under $H_{p^m}$ just by superscript $H$. For instance, we will write $D_p^H$ for $H_p$-invariants in $D_p$, and $Z^H$ for $H_p$-invariants in $Z$.

\subsection{Structure of the proof of Theorem \ref{main}}

Since the proof of Theorem~\ref{main} is rather technical, let us describe its structure.
The proof consists of three parts. To begin, we take a prime number $p \gg 0$. \par \smallskip

1. In Lemma \ref{saturation}, we show that all $H_p$-invariants in $D_p$ lift modulo $p^m$ for all $m$
(to invariants in $D_{p^m}$). This is done by induction in $m$, and is based
on Lemma \ref{generation}. The argument relies on constructing a large amount of invariants in characteristic zero (which is done in Lemma \ref{express}) and then reducing them modulo $p^m$. This creates a sufficient supply of invariants modulo $p^m$ to show that all invariants modulo $p^{m-1}$ must lift modulo $p^m$. \par \smallskip

2. Using Lemma \ref{saturation}, we show that the centralizer of $D_{p^m}^H$ in $D_{p^m}$ reduces to the center $Z_m$ of $D_{p^m}$. (Basically, the argument says that since there are a lot of invariants, commuting with them is a strong condition and forces the element to be in the center).
Using this, and Lemma \ref{center} (which says that the reduction modulo $p$ of $Z_m$ is $Z^{p^m}$), we prove in Proposition \ref{invariance} that $Z^{p^m}$ is $H_p$-invariant. \par \smallskip

3. Now Propositions~\ref{propCEW}(i) and~\ref{invariance} imply that the assumptions of Theorem~\ref{tower} applied to the $H_p$-action on $Z$
are satisfied. Applying Theorem~\ref{tower}, we conclude that $H_p$ is cocommutative.
Since this holds for sufficiently large $p$, we conclude that $H$ is cocommutative and hence a group algebra.

\subsection{Abundance of invariants}

By \cite[Theorem~2.2]{SVO}, the action of $H_{p^m}$ on $A_{p^m}$ extends to $D_{p^m}$.
The goal of this subsection is to show that there are ``many''  invariants of the $H_{p^m}$-action on $D_{p^m}$ for any $m$, in the sense that any invariant modulo $p^{m-1}$ lifts modulo $p^{m}$. \par \smallskip

We need the following notation. \medskip

\noindent $\bullet$ $D_p^H(m)=D_{p^m}^H/(pD_{p^m}\cap D_{p^m}^H)$,
identified with the image of $D_{p^m}^H$ in $D_p$. \smallskip

\noindent $\bullet$ $Z^H(m)$ is  the center of $D_p^H(m)$.
\medskip

\noindent Note that $D_p^H(m)$ is a division subalgebra of $D_p$, and
$$
D_p^H=D_p^H(1)\supset D_p^H(2)\supset \dots \supset D_p^H(m)\supset\dots.
$$

\begin{lemma}\label{generation} Take $p \gg 0$. Then, for any $m$, one has $D_p=D_p^H(m)A_p$, and hence $D_{p^m}=D_{p^m}^HA_{p^m}$. Moreover, $D_p$ is spanned over $D_p^H(m)$ as a left vector space by the monomials in $x_i,y_i$ of degree less or equal than $N$.
\end{lemma}

\begin{proof}
First note that if $p$ is large enough and $\psi$ is sufficiently generic, then the elements $b_{w,u}$ from Lemma~\ref{express} (for ${\rm deg}(w)=N+1,{\rm deg}(u)=N$) can be reduced modulo $p^m$ (cf. \cite[proof of Proposition~4.4]{CEW}). More precisely, we have $b_{w,u}=T^{-1}b_{w,u}'$, where $T,b_{w,u}'\in A$. We should choose $R$ so that it contains the coefficients of
$T,b_{w,u}'$. Then for sufficiently large $p$ and a suitably generic choice of $\psi$ the reduction of $T$ modulo $p$ is not zero, so the reduction of $T^{-1}$ is defined.  Let $b_{w,u,p^m}\in D_{p^m}$ be the reductions of $b_{w,u}$ modulo $p^m$. 

\vskip .05in

\noindent {\bf Sublemma.} One has $b_{w,u,p^m}\in D_{p^m}^H$, i.e., $b_{w,u,p^m}$ is invariant under $H_{p^m}$. 

\vskip .05in

\noindent {\it Proof of the Sublemma.} Let $b:=b_{w,u}$, and write $b$ as $T^{-1}a$, where $T,a\in A$. 
Then $Tb=a$. Since $b$ is $H$-invariant, applying the coaction 
to this equality, we obtain 
$$
\sum_i T_ib\otimes h_i^*=\sum_i a_i\otimes h_i^*,
$$
where $h_i$ is a basis of $H$, $h_i^*$ the dual basis of $H^*$, and 
$$
\rho(T)=\sum_i T_i\otimes h_i^*,\ \rho(a)=\sum_i a_i\otimes h_i^*. 
$$
Thus, $T_ib=a_i$ for all $i$. Since $A$ is an Ore domain, 
there exist $T_*\ne 0,a_*\in A$ such that $aT_*=Ta_*$.  
So $b=a_*T_*^{-1}$, hence $T_ia_*=a_iT_*$. 

For sufficiently large $p$, the reductions of 
all the above elements modulo $p^m$ are defined,
and the reduction of $T_*$ is invertible (i.e., nonzero modulo $p$). So we have the identities 
$$
a_{p^m}T_{*,p^m}=T_{p^m}a_{*,p^m},\ T_{i,p^m}a_{*,p^m}=a_{i,p^m}T_{*,p^m}, 
$$
$$
\rho(T_{p^m})=\sum_i T_{i,p^m}\otimes h_{i,p^m}^*,\ \rho(a_{p^m})=\sum_i a_{i,p^m}\otimes h_{i,p^m}^*, 
$$
where the subscripts $p^m$ denote the reductions modulo $p^m$. Thus, 
$$
\rho(T_{p^m})(a_{*,p^m}\otimes 1)=\rho(a_{p^m})(T_{*,p^m}\otimes 1).
$$
Hence, 
$$
\rho(T_{p^m})(a_{*,p^m}T_{*,p^m}^{-1}\otimes 1)=\rho(a_{p^m}).
$$
Therefore, 
$$
\rho(T_{p^m})(T_{p^m}^{-1}a_{p^m}\otimes 1)=\rho(a_{p^m}),
$$
or 
$$
T_{p^m}^{-1}a_{p^m}\otimes 1=\rho(T_{p^m}^{-1})\rho(a_{p^m})=\rho(T_{p^m}^{-1}a_{p^m}).
$$
This shows that the element $b_{p^m}=T_{p^m}^{-1}a_{p^m}$ is $H_{p^m}$-invariant, as desired. \qed
\medskip

By the Sublemma,  the elements $b_{w,u,p}$ belong to $D_p^H(m)$ for all $m$ (as they are reductions of $b_{w,u,p^m}$ modulo $p$).
So we conclude that $D_p^H(m)A_p$ is spanned over $D_p^H(m)$ by the monomials in $x_i$, $y_i$ of degree less or equal than $N$. Thus, $D_p^H(m)A_p$ is finite dimensional over $D_p^H(m)$, and hence the result follows from Lemma \ref{locali1}.
\end{proof}

Let $M$ be a free $\Bbb Z/p^m\Bbb Z$-module. Recall that a submodule $M'\subset M$
is called {\it saturated} if the natural map $M'/pM'\to M/pM$ is injective, that is,
$(pM) \cap M'=pM'$. Equivalently,  $M'$ is saturated if $M/M'$ is free.

\begin{example} The center $Z_m$ of $D_{p^m}$ is not saturated. By Lemma \ref{centerstr}, $Z_m$ contains elements $px_i^{p^{m-1}}$ which project to zero in $D_{p^m}/(p)$, but to nonzero in $Z_m/(p)$.
\end{example}

\begin{lemma}\label{saturation}  Take $p \gg 0$.
For any $m$, the inclusion $D_p^H(m) \hookrightarrow D_p^H$ is an isomorphism.
In other words, the $\Bbb Z/p^m\Bbb Z$-submodules $D_{p^m}^H\subset D_{p^m}$
are saturated (i.e., invariants modulo $p^{m-1}$ lift modulo $p^m$).
\end{lemma}

\begin{proof}
The degree of the division algebra $D_p^H(m)$
must be $p^s$ for some $s\le n$, since it is contained in
the division algebra $D_p$ which has degree $p^n$.
If $s<n$, then by Lemma \ref{divalg}(i), $[D_p:D_p^H(m)]$ has to be at least $p$.
But by Lemma~\ref{generation}, we have
$$
[D_p: D_p^H(m)]\le 1+2n+(2n)^2+\dots+(2n)^N,
$$
which is less than $p$ for $p$ sufficiently large. This means that for $p\gg 0$, we have $s=n$ and thus the degree of $D_p^H(m)$ is $p^n$. That is, the degrees of $D_p$ and $D_p^H(m)$, $m\ge 1$, including $D_p^H(1)=D_p^H$, are all the same. Thus, by Lemma~\ref{divalg}(ii), we have $Z^H(m-1)\supset Z^H(m)$ for all $m\ge 2$, and
$$
D_p^H(m-1) \cong Z^H(m-1)\otimes_{Z^H(m)}D_p^H(m).
$$
Hence, for $m \geq 1$,
\begin{equation} \label{eq:sat}
D_p^H=D_p^H(1) \cong Z^H(1)\otimes_{Z^H(m)}D_p^H(m)=Z^H \otimes_{Z^H(m)}D_p^H(m).
\end{equation}
Moreover,
\begin{equation} \label{eq:lessthanp}
[Z^H:Z^H(m)]=[D_p^H:D_p^H(m)]<p;
\end{equation}
this inequality holds as $p$ is sufficiently large.
\par \smallskip

Now let us prove that $D_p^H(m)=D_p^H$ by induction in $m$.
The statement for $m=1$ is trivial, so we may assume that $m\ge 2$ and the statement is known below $m$. \par \smallskip

Consider the spectral sequence attached to the filtration by powers of $p$ to compute the associated graded space of the cohomology of $H_{p^m}$ with coefficients in $D_{p^m}$ (in particular, of the zeroth cohomology, which is $D_{p^m}^H$). The $E_2$ page of this spectral sequence is defined by $E_2^{i,j}=H^i(H_p,D_p)$, and our job is to show that it degenerates at $E_2$ for $i=0$, i.e., that the differentials $d_1, \dots, d_{m-1}$
 vanish for $i=0$. By the induction assumption, the differentials
$$
d_1,\dots,d_{m-2}: D_p^H\to H^1(H_p,D_p)={\rm Ext}^1_{H_p}(\overline{\Bbb F}_p,D_p)
$$
are zero. Further,
we have a differential
$$
\partial :=d_{m-1}: D_p^H\to H^1(H_p,D_p).
$$
The restriction of $\partial$ to $Z^H$ is a derivation of $Z^H$ into the module
$H^1(H_p,D_p)$. Moreover, ${\rm Ker}(\partial |_{Z^H})=Z^H(m)$. (Indeed, for $z \in  Z^H$, $d_{m-1}(z)$ characterizes the failure of $z$ to lift modulo $p^m$ when it is known to lift modulo $p^{m-1}$.)

\par \smallskip

Now take $z\in Z^H$, and let its minimal polynomial over $Z^H(m)$ be $P$.
So, we obtain $0=\partial P(z)=P'(z)\partial(z)$. Since $[Z^H:Z^H(m)]<p$ by \eqref{eq:lessthanp}, we have $\deg(P)<p$. So, $P'(z)\ne 0$  and we get that $\partial(z)=0$. Thus, $Z^H(m)={\rm Ker}(\partial |_{Z^H}) = Z^H$, and hence $D_p^H(m)=D_p^H$ by \eqref{eq:sat}.
\end{proof}

\subsection{Invariance of $Z^{p^m}$ under the action of $H_p$.}

Suppose that $p\gg 0$.

 \begin{lemma}\label{central} The centralizer of $D_{p^m}^H$ in $D_{p^m}$ coincides with $Z_m$. As a consequence, $Z_m$ is $H_{p^m}$-stable.
 \end{lemma}

 \begin{proof}
 Let $u\in D_{p^m}$ be such that $[D_{p^m}^H,u]=0$. The map $D_{p^m}\to D_{p^m}$ given by
$a\mapsto [a,u]$ is a derivation of $D_{p^m}$.  By way of contradiction, suppose that this derivation is nonzero. Let $r$ be the largest integer such that \linebreak $[D_{p^m},u]\subset p^rD_{p^m}$. Then, $[?,u]$ defines a nonzero map $\partial: D_p\to D_p$, such that $\partial a$ is the image of $[\widetilde{a},u]$ in $p^rD_{p^m}/p^{r+1}D_{p^m}\cong D_p$ for any lift $\widetilde{a}$ of $a$ to $D_{p^m}$. It is clear that $\partial$ is a derivation, so $\partial(Z)\subset Z$. Also, by Lemma \ref{saturation}, $\partial(D_p^H)=0$. \par \smallskip

From Proposition~\ref{propCEW}(i) we obtain $ZD_p^H=D_p$. Thus,
to get a contradiction, it suffices to show that $\partial(Z)=0$. Let $z\in Z$, and $P$ be the minimal polynomial of $z$ over $Z^H$. Since $\partial(Z^H)=0$, we have $0=\partial P(z)=P'(z)\partial z$. Since $p \gg 0$, $[Z:Z^H]\le \dim H <p$, and hence, $P' \neq 0$. Thus,  $\partial z=0$, which gives the desired contradiction. \par \smallskip

The last statement follows since $(h \cdot z) a = a(h \cdot z)$, for $h \in H_{p^m}$, $a \in D_{p^m}^H$, and for $z$ in the centralizer of $D_{p^m}^H$ in $D_{p^m}$.
 \end{proof}

\begin{lemma}\label{center}
The image of $Z_m$ in $D_p$ is $Z^{p^{m-1}}$.
\end{lemma}

\begin{proof} This is a straightforward calculation with the Weyl algebra.
Namely, recall the subring $K_m\subset Z_m$ defined in Subsection \ref{centerst}.
It follows from Lemma \ref{centerstr} that $Z_1=Z=K_1$ and $Z_m=K_m+pZ_{m-1}$ for $m\ge 2$.
But
\begin{equation*}
Z^{p^m}=\overline{\Bbb F}_p\left(x_i^{p^{m+1}},y_i^{p^{m+1}}~:~ i=1,\dots,n\right),
\end{equation*}
hence $K_m$ projects surjectively onto $Z^{p^{m-1}}$ under reduction modulo $p$.
This implies the statement.
\end{proof}

\begin{proposition}\label{invariance}
The $H_p$-action on $Z$ preserves $Z^{p^m}$ for all $m$.
\end{proposition}

\begin{proof}
It follows from Lemma \ref{central} that $H_{p^m}$ preserves $Z_m$. Therefore, by Lemma \ref{center},
$H_p$ preserves $Z^{p^m}$.
\end{proof}

\subsection{Proof of Theorem~\ref{main}}\label{proofsection} Let $p\gg 0$.
By Proposition~\ref{propCEW}(ii), $H_p$ acts inner faithfully on $Z$. Therefore,
by Proposition \ref{invariance}, the assumptions of Theorem \ref{tower}
applied to the $H_p$-action on $Z$ are satisfied. So by Theorem~\ref{tower},
$H_p$ is cocommutative (a group algebra). But by
\cite[Lemma 2.3(ii)]{CEW}, the product of all $\psi$ for $p\ge \ell$ is injective for any $\ell$, so we conclude that $H_R$ is cocommutative.
Hence, $H$ is cocommutative. Thus, $H=kG$, where $G$ is a finite group, and
Theorem~\ref{main} is proved.  \qed

\subsection{Proof of Theorem~\ref{diffop}} The proof is parallel to that of Theorem~\ref{main}, and obtained by replacing $A=\bold A_n(k)$ by $A=D(X)$, using the fact that the reduction of $X$ mod $p$ is smooth for large $p$ and generic $\psi$ 
\cite[17.7.8(ii)]{EGA}. 

Let us list the necessary changes. \par \smallskip

1. In Lemma \ref{express} and below, $x_i,y_i$ should be replaced by any finite set of generators $L_1,\dots,L_r$ of $D(X)$, and the number $2n$ in the proof of Lemma~\ref{saturation} should be replaced by $r$. \par \smallskip

2. The discussion in Subsection \ref{centerst} should be modified as follows.
Pick a point $x\in X$, and let $x_1,\dots,x_n$ be local coordinates near $x$.
Let $y_i=\frac{\partial}{\partial x_i}$ be the corresponding partial
derivatives; they are rational vector fields on $X$.
Let $f_1,\dots,f_q$ be generators of the algebra of regular functions $\mathcal{O}(X)$ on $X$. Let $K_m=W_{m,p}(f_i^{p^m},y_j^{p^m})$, where reductions of $f_i,y_j$ modulo $p^m$ are also denoted by $f_i,y_j$, respectively.
Then, one can check by computing in local coordinates that $Z_m=K_m+pZ_{m-1}$, so that the proof of Lemma~\ref{center} goes through. \qed

\section*{Acknowledgments}
\noindent

We thank the referee for a careful reading of this manuscript, which led to the addition of the sublemma in the proof of Lemma~4.1. We are also grateful to B. Poonen for a precise reference to [EGA].
J. Cuadra was supported by grant MTM2014-54439 from MINECO and FEDER and by the research group FQM0211 from Junta de Andaluc\'{\i}a. P. Etingof and C. Walton were supported by the US National Science Foundation, grants DMS-1000113,1502244 and DMS-1550306, respectively. \par \smallskip

This work was mainly done during the visit of the first named author to the Mathematics department of MIT. He is deeply grateful for the kind and generous hospitality and stimulating atmosphere. This visit was financed through grant PRX14/00283 from the Spanish Programme of Mobility of Researchers.

\end{document}